\documentclass{amsart}
\usepackage{amsmath, amssymb, amscd, amsbsy, amsthm}
\usepackage{graphicx}
\newtheorem{theorem}{Theorem}[section]
\newtheorem{lemma}[theorem]{Lemma}

\theoremstyle{definition}

\theoremstyle{remark}

\numberwithin{equation}{section}

\def\Z{{\mathbb Z}}

\def\R{{\mathbb R}}

\def\volm{{\rm Diff}_{\Omega}^{\infty} (M)_0}

\def\uvolm{\widetilde{{\rm Diff}_{\Omega}^{\infty}} (M)_0}
\def\Flux{{\rm Flux}}
\def\uFlux{\widetilde{\rm Flux}}
\def\Hom{{\rm Hom}}
\def\rank{{\rm rank}}

\begin{document}
\title[Homomorphisms on volume-preserving diffeomorphisms]
{Homomorphisms on groups of volume-preserving diffeomorphisms 
via fundamental groups}

\author{Tomohiko Ishida}
\address{Department of Mathematics, 
Kyoto University, 
Kitashirakawa Oiwake-cho, Sakyo-ku, Kyoto 606-8502, Japan.}
\email{ishidat@math.kyoto-u.ac.jp}

\dedicatory{Dedicated to Professor Takashi Tsuboi 
on the occasion of his $60$-th birthday}
\subjclass[2010]{37C15}
\keywords{volume-preserving diffeomorphisms, flux homomorphism, flux groups}
\date{\today}
\begin{abstract}
Let $M$ be a closed manifold. 
Polterovich constructed 
a  linear map from the vector space 
of quasi-morphisms on the fundamental group $\pi _{1}(M)$ of $M$ 
to the space of quasi-morphisms 
on the identity component $\volm$ 
of the group of volume-preserving diffeomorphisms of $M$. 
In this paper, 
the restriction 
$H^{1}(\pi_{1}(M); \R )\to H^{1}(\volm ; \R)$ 
of the linear map 
is studied and 
its relationship with the flux homomorphism 
is described. 
\end{abstract}
	
\maketitle

%%%%%%%%%%%%%%%%%%%%%%%%%%%%%%%%%%%%%%%%%%%%%%%%%%%%%%%%%%%%%%%%%%%%%%%%%%%%%%%%%%%%%%%%%%%%%%%%%%%%%%%%%%%%%%%%%%%%%%%%%%%%%%%%%%%%%%%%%
\section{Introduction}
Let $M$ be a closed connected Riemannian manifold 
and $\Omega$ a volume form on $M$. 
We denote by $\volm$ 
the identity component 
of the group of volume-preserving $C^{\infty}$-diffeomorphisms of $M$. 
We assume 
that the center of the fundamental group $\pi _{1}(M)$ is finite. 
In \cite{gg04}, Gambaudo and Ghys 
constructed countably many quasi-morphisms 
on the group of area-preserving diffeomorphisms of the $2$-disk 
from the signature quasi-morphism on the braid groups. 
After that Polterovich introduced in \cite{polterovich06}  
a similar construction of quasi-morphisms on $\volm$ 
from quasi-morphisms on $\pi _{1}(M)$. 
Recently, Brandenbursky generalized these strategy 
and defined a homomorphism 
from the vector space of quasi-morphisms 
on the braid group or the fundamental group
to the space of quasi-morphisms of volume-preserving diffeomorphisms 
\cite{brandenbursky13p}\cite{brandenbursky11}. 

Polterovich's construction 
induces a linear map 
from the vector space of quasi-morphisms on $\pi _{1}(M)$ 
to the vector space of quasi-mor-phisms on $\volm$.  
By restricting it on $H^{1}(\pi _{1}(M); \R )$, 
we have the linear map 
$\Gamma\colon H^{1}(\pi _{1}(M); \R )\to H^{1}(\volm ;\R )$, 
which is defined in section 2 of this paper. 
Studying the linear map 
$\Gamma\colon H^{1}(\pi _{1}(M); \R )\to H^{1}(\volm ;\R )$, 
we have a sufficient condition for vanishing of the volume flux group 
which is first obtained by K\c{e}dra-Kotschick-Morita in another way. 

\begin{theorem}[K\c{e}dra-Kotschick-Morita\cite{kkm06}]\label{vanish}
If the center of $\pi _1(M)$ is finite, 
then the volume flux group of $M$ is trivial. 
\end{theorem}

\begin{sloppypar}
Let 
$\Flux\colon\volm\to H_{\rm dR}^{n-1}(M; \R )$
be the $\Omega$-flux homomorphism. 
Let $I^{k}\colon H_{\rm dR}^{k}(M; \R )\to H^{k}(M; \R )$
be the isomorphism 
which gives the identification 
of the de Rham cohomology and the singular cohomology  
defined by 
\[ I^{k}([\eta ])(\sigma )=\int _{\sigma }\eta \]
for $k$ dimensional closed differential form $\eta$ 
and for $k$-chain $\sigma$. 
Let $PD\colon H^{n-1}(M; \R )\to H_{1}(M; \R )$ 
be the Poincar\'{e} duality. 
Our main result is the following. 

\begin{theorem}\label{main}
For any $\phi\in H^{1}(\pi _{1}(M); \R )=H^{1}(M; \R )$,  
\[ \Gamma (\phi )
=\phi \circ PD\circ I^{n-1}\circ\Flux\colon\volm\to\R . \]
\end{theorem}
\end{sloppypar}

%%%%%%%%%%%%%%%%%%%%%%%%%%%%%%%%%%%%%%%%%%%%%%%%%%%%%%%%%%%%%%%%%%%%%%%%%%%%%%%%%%%%%%%%%%%%%%%%%%%%%%%%%%%%%%%%%%%%%%%%%%%%%%%%%%%%%%%%%
%%%%%%%%%%%%%%%%%%%%%%%%%%%%%%%%%%%%%%%%%%%%%%%%%%%%%%%%%%%%%%%%%%%%%%%%%%%%%%%%%%%%%%%%%%%%%%%%%%%%%%%%%%%%%%%%%%%%%%%%%%%%%%%%%%%%%%%%%
\section{preliminaries}

In this section, we define a linear map
\[ \Gamma\colon H^{1}(\pi _{1}(M); \R)\to H^{1}(\volm ;\R ) \] 
and recall a definition of the flux homomorphism. 

Here and throughout this paper, 
we use functional notation. 
That is, 
for any homotopy classes $\gamma _{1}$ and $\gamma _{2}$ 
of loops with a fixed base point, 
the multiplication $\gamma _{1}\gamma_{2}$ means 
that $\gamma _{2}$ is applied first. 

Choose a base point $x^{0}$ of $M$. 
For almost every $x\in M$, 
we choose the shortest geodesic 
$a_x\colon [0, 1]\to M$ 
connecting $x^{0}$ with $x$ 
if it is uniquely determined. 
For any $f\in\volm$ and almost every $x\in M$ 
for which both the geodesics $a_{x}$ and $a_{f(x)}$ is defined, 
we define the loop $l(f; x)\colon [0, 1]\to M$ by
\[ l(f; x)(t)=\left\{ \begin{array}{ll}
a_{x}(3t) & (0\leq t\leq\frac{1}{3}) \\[.6em]
f_{3t-1}(x) & (\frac{1}{3}\leq t\leq\frac{2}{3}) \\[.6em]
a_{f(x)}(3-3t) & (\frac{2}{3}\leq t\leq 1) \\
\end{array}\right. ,\]
where $\{ f_{t}\}_{t\in[0, 1]}$ is an isotopy 
such that $f_{0}$ is the identity and $f_{1}=f$. 
Of course for some $x\in M$ 
there exist two or more shortest geodesics connecting $x^{0}$ with $x$. 
However for almost every $x\in M$ the loop $l(f; x)$ is well-defined. 
We denote by $\gamma (f; x)$ the homotopy class represented 
by the loop $l(f; x)$. 
For a homomorphism $\phi\in H^{1}(\pi _{1}(M); \R )$, 
we define the homomorphism 
$\Gamma (\phi )\in H^{1}(\volm ; \R )$ by 
\begin{align*}
\Gamma (\phi )(f)
=\int _{x\in M}\phi (\gamma (f; x))\Omega . 
\end{align*}
For almost every $x\in M$, 
the homotopy class $\gamma (f; x)$ is well-defined 
and is unique up to elements of the center of $\pi _{1}(M)$ 
\cite{polterovich06}. 
Since the center of $\pi _{1}(M)$ 
is finite, 
the image of $\gamma (f; x)$ by the homomorphism 
$\phi\colon\pi _{1}(M; x^{0})\to\R$ 
is independent of the choice of the flow $\{ f_{t}\}$. 
Since the manifold $M$ is compact, 
the loops $l(f; x)$ have uniformly bounded length for fixed $\{ f_{t}\}$. 
Hence the map $\gamma (f; \cdot )\colon M\to\pi _{1}(M; x^{0})$ 
has a finite image 
and the value $\Gamma(\phi )(f)$ is well-defined. 

Let $\uvolm$ be the universal cover of $\volm$.
Consider a path $\{f_{t}\}_{t\in [0, 1]}$ in $\volm$ 
such that $f_{0}$ is the identity. 
Let $X_{t}$ be the corresponding vector field.
Then the map 
$\uFlux\colon\uvolm\to H_{\rm dR}^{n-1}(M; \R)$ 
is defined by 
\[ \uFlux (\{ f_{t}\})=\left[ \int_{0}^{1}\iota _{X_{t}}(\Omega )dt\right], \]
where $\iota _{X_{t}}$ is the interior product by $X_{t}$. 
The map 
$\uFlux\colon\uvolm\to H_{\rm dR}^{n-1}(M; \R)$ 
is a well-defined homomorphism 
and called the $\Omega$-{\it flux homomorphism}.
The fundamental group $\pi_{1}(\volm)$ 
is contained in $\uvolm$ as a subgroup of deck transformations. 
The image $G_{\Omega}=\uFlux(\pi_{1}(\volm))$ 
of $\pi_{1}(\volm)$ 
by the $\Omega$-flux homomorphism 
$\uFlux\colon\uvolm\to H_{\rm dR}^{n-1}(M; \R)$ 
is called the {\it volume flux group} of $M$  
and the homomorphism 
$\uFlux\colon\uvolm\to H_{\rm dR}^{n-1}(M; \R)$
descends to the homomorphism 
$\Flux\colon\volm\to H_{\rm dR}^{n-1}(M; \R)/G_{\Omega}$, 
which is also called the $\Omega$-flux homomorphism.  

%%%%%%%%%%%%%%%%%%%%%%%%%%%%%%%%%%%%%%%%%%%%%%%%%%%%%%%%%%%%%%%%%%%%%%%%%%%%%%%%%%%%%%%%%%%%%%%%%%%%%%%%%%%%%%%%%%%%%%%%%%%%%%%%%%%%%%%%%
%%%%%%%%%%%%%%%%%%%%%%%%%%%%%%%%%%%%%%%%%%%%%%%%%%%%%%%%%%%%%%%%%%%%%%%%%%%%%%%%%%%%%%%%%%%%%%%%%%%%%%%%%%%%%%%%%%%%%%%%%%%%%%%%%%%%%%%%%
\section{Proofs}
In this section, 
we give proofs of Theorems \ref{vanish} and \ref{main}. 
The following theorem is mentioned in \cite{polterovich06} 
without proof. 

\begin{theorem}\label{polterovich}
The linear map 
\[ \Gamma\colon H^{1}(\pi _{1}(M); \R )\to H^{1}(\volm ; \R ) \]
is injective. 
\end{theorem}

We give a proof of Theorem \ref{polterovich}. 
Let $\beta\in\pi_{1}(M; x^{0})$. 
Suppose that we can choose 
a loop $l$ representing $\beta$ without self-intersection. 
Choose a tubular neighborhood $N\subset M$  of $l$ 
and a diffeomorphism $\varphi\colon N\to D^{n-1}\times S^{1}$.  
Let $(z, s)$ be the coordinate on $D^{n-1}\times S^{1}$. 
We may assume that 
there exists $\Omega '\in A^{n-1}(D^{n-1}; \R )$
such that $\varphi^{\ast}(\Omega 'ds)=\Omega |_{N}$ 
by changing the neighborhood $N$ and diffeomorphism $\varphi$ 
if necessary.  
Let $\omega\colon D^{n-1}\to\R$ be a function 
such that $\omega (z)=0$ in a neighborhood of the boundary. 
We define the volume-preserving diffeomorphism $f_{\omega}$ 
of  $D^{n-1}\times S^{1}$ 
by 
\[ f_{\omega}(z, s)=(z, s+\omega (z)). \]
and define $F_{\omega}\in\volm$ to be the identity outside $N$
and $F_{\omega}=\varphi^{-1}f_{\omega}\varphi$ on $N$. 

\begin{lemma}\label{calgamma}
For any $\phi\in H^{1}(\pi_{1}(M); \R )$, 
\[ \Gamma (\phi)(F_{\omega})
=\phi(\beta)\int _{z\in D^{n-1}} \omega (z)\Omega'.  \]
\end{lemma}

\begin{proof}
Note that the base point $x^{0}$ of $M$ is in $N$. 
Let us denote $\varphi(x^{0})$ by $(z^{0}, s^{0})$
and $\varphi(x)$ by $(z^{1}, s^{1})$. 
Let $v$ be the smallest non-negative number 
such that $s^{1}+v=s^{0}$. 
For each $x\in N$  
we define the paths $l_{1}, l_{2}, l_{3}\colon [0, 1]\to D^{n-1}\times S^{1}$ by
\begin{align*}
l_{1}(t)&=(tz^{0}+(1-t)z^{1}, s^{1}), \\
l_{2}(t)&=(z^{0}, s^{1}+tv), \\
l_{3}(t)&=(z^{1}, s^{1}+t(\omega(z^{1})-[\omega(z^{1})])). 
\end{align*}
We define the homotopy classes $\zeta_{x}, \eta_{x}$ of loops in $M$ by 
\[ \zeta_{x}=[(\varphi^{-1})_{\ast}(l_{2}l_{1})a_x], 
\quad 
\eta_{x}=[a_{F_{\omega}(x)}^{-1}(\varphi^{-1})_{\ast}(l_{3})a_{x}]. \]

Since the path $\{F_{t\omega}\}$ connects the identity and $F_{\omega}$ in $\volm$, 
the homotopy class $\gamma (F_{\omega}; x)$ is trivial if $x\not\in N$. 
On the other hand, 
$\gamma (F_{\omega}; x)$ can be written as 
\[ \gamma (F_{\omega}; x)=\eta_{x}\zeta_{x}^{-1}\beta^{[\omega (z')]}\zeta_{x} \]
if $x\in N$. 
Therefore, 
\begin{align*}
\Gamma (\phi)(F_{\omega})
&=\int _{x\in N} \phi (\gamma (F_{\omega}; x))\Omega \\
&=\phi (\beta)\int _{x\in N}[\omega (z')]\Omega 
+\int _{x\in N} \phi(\eta_{x})\Omega .
\end{align*}

Since $F_{\omega}^{k}=F_{k\omega}$ 
for any $k\in\Z$, 
\[ \Gamma (\phi)(F_{\omega})
=\lim _{k\to\infty}\frac{1}{k}\Gamma (\phi)(\gamma (F_{k\omega}; x))\Omega.  \]
Since the domain $N$ is compact, the value $\phi(\eta_{x})$ is bounded 
and thus we have
\begin{align*}
\Gamma (\phi)(F_{\omega})
&=\phi(\beta)\int _{x\in N} \omega (z)\Omega \\
&=\phi(\beta)\int _{z\in D^{n-1}} \omega (z)\Omega'. 	
\end{align*}
\end{proof}

\begin{proof}[Proof of Theorem\ref{polterovich}]
Suppose 
a homomorphism $\phi\in H^{1}(\pi_{1}(M); \R )$ 
is non-trivial. 
Then there exists a homotopy class $\beta$ of a loop 
without self-intersection in $M$ 
such that $\phi (\beta)\neq 0$. 
It is sufficient to prove that 
there exists $g\in\volm$ 
such that $\Gamma (\phi)(g)\neq 0$. 
If we choose a function 
$\omega\colon D^{n-1}\to\R$ 
such that 
\[ \int_{z\in D^{n-1}}\omega(z)\Omega '\neq 0, \]
then by Lemma \ref{calgamma} we have 
$\Gamma (\phi)(F_{\omega})\neq 0$. 
\end{proof}

\begin{proof}[Proof of Theorem \ref{vanish}]
It is known 
that the flux homomorphism 
gives the abelianization 
of the group $\volm$ \cite{banyaga}. 
Hence 
for any homomorphism 
$\phi\in H^{1}(\pi _{1}(M); \R)$ 
there exists a homomorphism 
\[ A_{\phi}\colon H_{\rm dR}^{n-1}(M; \R)/G_{\Omega}\to\R \]
such that 
the homomorphism $\Gamma (\phi)\in H^{1}(\volm ; \R )$ 
can be represented by the composition of homomorphisms 
$\Flux\colon\volm\to H_{\rm dR}^{n-1}(M; \R )/G_{\Omega}$ 
and  
$A_{\phi}\colon H_{\rm dR}^{n-1}(M; \R )/G_{\Omega}\to\R$. 
That is, 
\[ \Gamma (\phi)=A(\phi)\circ\Flux\colon\volm\to\R. \]
Since the diffeomorphism $F_{\omega}$ is the time $1$-map 
of the time independent vector field 
\[ X_{x}=
\left\{ \begin{array}{ll}
(\varphi^{-1})_{\ast}(\omega (z)\frac{d}{ds}) 
& \text{if }x\in N \\
0 & \text{if }x\not\in N \\
\end{array}\right. , \]
we have 
\[ \Flux (F_{\omega})=\iota _{X}\Omega 
=\varphi^{\ast}[\omega (z)\Omega ']. \]
In particular, 
\[ \Flux (F_{s\omega})=s\Flux (F_{\omega}) \]
for any $\beta\in\pi _{1}(M)$,  
any function $\omega\colon D^{n-1}\to\R$ 
and any $s\in\R$. 
On the other hand 
by Lemma \ref{calgamma} 
\[\Gamma (\phi )(F_{t\omega})
=t\Gamma (\phi )(F_{\omega}) \]
for any $t\in\R$. 
Choose elements
$\beta _{1}, \dots , \beta _{m}\in\pi_{1}(M, x^{0})$ 
whose images by the projection 
$\pi_{1}(M, x^{0} )\to H_{1}(M; \Z )$ 
form a basis of $H_{1}(M; \R )$. 
If we replace $\beta$ with $\beta _{1}, \dots , \beta _{m}$, 
then $(n-1)$-classes $\varphi^{\ast}[\omega (z)\Omega ']$'s 
form a basis of $H_{dR}^{n-1}(M; \R )$.
Hence if there exists a non-trivial element $\xi\in G_{\Omega}$, 
then $A_{\phi}(t\xi)=0$ 
for any $t\in\R$. 
The map $A_{\phi}$ descends to the linear map 
$A'_{\phi}\colon H_{dR}^{n-1}(M; \R )/\langle G_{\Omega}\rangle\to \R$, 
where $\langle G_{\Omega}\rangle$ 
means the vector subspace of  $H_{dR}^{n-1}(M; \R )$ 
spanned by elements of $G_{\Omega}$. 

By Theorem \ref{polterovich}, 
\[ \rank_{\R}H^{1}(M; \R)=\rank_{\R}{\rm Im}\Gamma
\leq \rank_{\R}\Hom(H_{\rm dR}^{n-1}(M; \R)/\langle G_{\Omega}\rangle, \R). \]
If there exists a non-trivial element $\xi\in G_{\Omega}$, 
then
\[ \rank_{\R}\Hom(H_{\rm dR}^{n-1}(M; \R)/\langle G_{\Omega}\rangle , \R)
< \rank_{\R}H^{n-1}(M; \R). \]
while by the Poincar\'{e} duality 
\[ \rank_{\R}H^{1}(M; \R)=\rank_{\R}H^{n-1}(M; \R). \]
This contradiction shows 
that there's no non-trivial element in $G_{\Omega}$. 
\end{proof}

\begin{proof}[Proof of Theorem \ref{main}]
The statement is that 
\[ A_{\phi}=\phi\circ PD\circ I^{n-1}
\colon H_{dR}^{n-1}(M; \R )\to\R. \] 
Since $A_{\phi}\colon H_{dR}^{n-1}(M; \R )\to\R$ 
is a linear map, 
it is sufficient to 
choose  $\eta_{1}, \dots , \eta_{m}$ 
generating $H_{\rm dR}^{n-1}(M; \R )$
and prove that 
$A_{\phi}(\eta_{i})=\phi\circ PD\circ I^{n-1} (\eta _{i})$ for $1\leq i\leq m$. 

Since 
\[ \Flux (F_{\omega})=\iota _{X}\Omega 
=\varphi^{\ast}[\omega (z)\Omega '], \]
we have 
\[ I^{n-1}\circ\Flux(F_{\omega})(\sigma )
=\int _{\varphi_{\ast}\sigma} \omega (z)\Omega '. \]
Therefore, 
\[ PD\circ I^{n-1}\circ\Flux(F_{\omega})
=\left(\int _{z\in D^{n-1}}\omega (z)\Omega '\right)\beta. \]
Comparing this equation with Lemma \ref{calgamma}, 
we have 
\[ \Gamma (\phi)(F_{\omega}) 
=\phi\circ PD\circ I^{n-1}\circ\Flux (F_{\omega}) \]
for any $\phi\in H^{1}(M; \R )$. 

As in the proof of Theorem \ref{vanish}, 
choose homotopy classes 
$\beta _{1}, \dots , \beta _{m}\in\pi_{1}(M, x^{0})$ 
whose images by the projection
$\pi_{1}(M x^{0} )\to H_{1}(M; \Z )$ 
form a basis of $H_{1}(M; \R )$. 
If we replace $\beta$ with $\beta _{1}, \dots , \beta _{m}$, 
then $\Flux (F_{\omega })$'s 
form a basis of $H_{dR}^{n-1}(M; \R )$ 
and hence 
this completes the proof. 
\end{proof}

%%%%%%%%%%%%%%%%%%%%%%%%%%%%%%%%%%%%%%%%%%%%%%%%%%%%%%%%%%%%%%%%%%%%%%%%%%%%%%%%%%%%%%%%%%%%%%%%%%%%%%%%%%%%%%%%%%%%%%%%%%%%%%%%%%%%%%%%%
\providecommand{\bysame}{\leavevmode\hbox to3em{\hrulefill}\thinspace}
\providecommand{\MR}{\relax\ifhmode\unskip\space\fi MR }
\providecommand{\MRhref}[2]{%
  \href{http://www.ams.org/mathscinet-getitem?mr=#1}{#2}
}
\providecommand{\href}[2]{#2}

\end{document}